\documentclass{amsart}

\usepackage{amsmath, amsthm, amssymb,mathtools}
\usepackage{mathrsfs}
\usepackage[all]{xy}
\usepackage{paralist}
\usepackage[vcentermath]{youngtab}
\theoremstyle{definition}

\makeatletter
\newtheorem*{rep@theorem}{\rep@title}
\newcommand{\newreptheorem}[2]{%
\newenvironment{rep#1}[1]{%
 \def\rep@title{#2~\ref{##1}}%
 \begin{rep@theorem}}%
 {\end{rep@theorem}}}
 
\makeatother
\newtheorem{theorem}{Theorem}%[section]
\newtheorem{definition}[theorem]{Definition}

\newtheorem{proposition}[theorem]{Proposition}

\newtheorem{conjecture}[theorem]{Conjecture}

\newreptheorem{theorem}{Theorem}
\newreptheorem{proposition}{Proposition}

\theoremstyle{remark}
\newtheorem{remark}[theorem]{Remark}
\newtheorem{example}[theorem]{Example}

\begin{document}

\title{Enumeration of points, lines, planes, etc.}

\author{June Huh and Botong Wang}

\address{Institute for Advanced Study, Fuld Hall, 1 Einstein Drive, Princeton, NJ, USA.}
\address{Korea Institute for Advanced Study,
85 Hoegiro, Dongdaemun-gu, Seoul 130-722, Korea}
\email{huh@princeton.edu}

%\address{Korea Institute for Advanced Study, Dongdaemun-gu, Seoul 130-722, Republic of Korea.}
%\email{junehuh@kias.re.kr}

\address{University of Wisconsin-Madison, Van Vleck Hall, 480 Lincoln Drive, Madison, WI, USA.}
\email{bwang274@wisc.edu}

\maketitle

\section{Introduction}

One of the earliest results in enumerative combinatorial geometry is the following theorem of de Bruijn and Erd\H{o}s \cite{deBruijn-Erdos}:
\begin{quote}
\emph{Every  finite set of points $E$ in a projective plane determines at least $|E|$  lines, unless $E$ is contained in a line.}
\end{quote}
In other words, if $E$ is not contained in a line, then the number of lines in the plane containing at least two points in $E$ is at least $|E|$.
See \cite{deWitteI,deWitteII} for an interesting account of its history and a survey of known proofs.

The following more general statement, conjectured by Motzkin in \cite{MotzkinThesis}, 
was subsequently proved by many in various settings:
\begin{quote}
\emph{Every finite set of points $E$ in a projective space determines at least $|E|$  hyperplanes, unless $E$ is contained in a hyperplane.}
\end{quote}
 Motzkin proved the above for $E$ in real projective spaces \cite{Motzkin}.  
Basterfield and Kelly  \cite{Basterfield-Kelly} showed the statement  in general,
and Greene \cite{Greene} 
strengthened the result  by showing that there is a \emph{matching} from $E$  to the set of hyperplanes determined by $E$, unless $E$ is contained in a hyperplane:
\begin{quote}
\emph{For every point in $E$ one can choose a hyperplane containing the point in such a way that no hyperplane is chosen twice.}
\end{quote}
Mason \cite{Mason} and Heron \cite{Heron} obtained similar results by different methods.

Let $\mathbb{P}$ be the projectivization of an $r$-dimensional vector space over a field,
  $E \subseteq \mathbb{P}$ be a finite subset not contained in any hyperplane,
 and  $\mathscr{L}$ be the poset of subspaces of $\mathbb{P}$ spanned by the subsets of $E$.
The poset $\mathscr{L}$ is a graded lattice, and its rank function  satisfies the submodular inequality
\[
\text{rank}(F_1)+\text{rank}(F_2) \ge \text{rank}(F_1 \lor F_2)+\text{rank}(F_1 \land F_2) \ \ \text{for all $F_1,F_2 \in \mathscr{L}$}.
\]
For a nonnegative integer $p$, we write $\mathscr{L}^p$ for the set of rank $p$ elements in the lattice $\mathscr{L}$.
Thus $\mathscr{L}^1$ is the set of points in $E$,  $\mathscr{L}^2$ is the set of lines joining points in $E$, and $\mathscr{L}^r$ is the set with one element, $\mathbb{P}$.
Graded posets obtained in this way are standard examples of \emph{geometric lattices} \cite{Welsh}.
These include the lattice of all subsets of a finite set (Boolean lattices), the lattice of all partitions of a finite set (partition lattices), and the lattice of all subspaces of a finite vector space (projective geometries).
In \cite{Dowling-WilsonII}, Dowling and Wilson further generalized the above results for geometric lattices:
\begin{quote}
\emph{For every nonnegative integer $p$ less than $\frac{r}{2}$, 
there is a matching from the set of rank at most $p$ elements of $\mathscr{L}$ to the set of corank at most $p$ elements of $\mathscr{L}$.}
\end{quote}
The matching can be chosen to match the minimum of $\mathscr{L}$ to the maximum of $\mathscr{L}$, and hence the above statement covers all the results introduced above.
Kung gave another proof of the same result from the point of view of Radon transformations in \cite{KungRadonI,KungRadon}.

In \cite{Dowling-WilsonI,Dowling-WilsonII}, Dowling and Wilson stated the following ``top-heavy'' conjecture.

\begin{conjecture}\label{MainConjecture}
Let $\mathscr{L}$ be a geometric lattice of rank $r$.
\begin{enumerate}[(1)] \itemsep 5pt
\item For every nonnegative integer $p$ less than $\frac{r}{2}$, 
\[
|\mathscr{L}^p| \le |\mathscr{L}^{r-p}|.
\]
In fact, there is an injective map $\iota:\mathscr{L}^p \to \mathscr{L}^{r-p}$ satisfying $x \le \iota(x)$ for all $x$.
\item For every nonnegative integer $p$ less than $\frac{r}{2}$, 
\[
|\mathscr{L}^p| \le |\mathscr{L}^{p+1}|.
\]
In fact, there is an injective map $\iota: \mathscr{L}^p \to \mathscr{L}^{p+1}$ satisfying $x \le \iota(x)$ for all $x$.
\end{enumerate}
\end{conjecture}

The conjecture was reproduced in  \cite[Exercise 3.37]{StanleyEC} and \cite[Exercise 3.5.7]{Kung-Rota-Yan}.
For an overview and related results, see \cite{Aigner}.
When $\mathscr{L}$ is a Boolean lattice or a  projective geometry, the validity of Conjecture \ref{MainConjecture} is a classical result. We refer to \cite{LefschetzBook} and \cite{StanleyAC} for recent expositions. In these cases, Conjecture \ref{MainConjecture} implies that  $\mathscr{L}$ has the \emph{Sperner property}:
\begin{quote}
\emph{The maximal number of incomparable elements in $\mathscr{L}$ is the maximum of $|\mathscr{L}^p|$ over  $p$.}
\end{quote}
Kung proved the second part of Conjecture \ref{MainConjecture} for partition lattices in \cite{KungRadonII}.
Later he showed the second part of Conjecture \ref{MainConjecture} for $p\le 2$
 when  every  line contains the same number of points  \cite{KungLinesPlanes}.

We now state our main result. As before, we write $\mathbb{P}$ for the projectivization of an $r$-dimensional vector space over a field.

\begin{theorem}\label{MainTheoremI}
Let $E \subseteq \mathbb{P}$ be a finite subset not contained in any hyperplane, and $\mathscr{L}$ be the poset of subspaces of $\mathbb{P}$ spanned by subsets of $E$.
\begin{enumerate}[(1)]\itemsep 5pt
\item For all nonnegative integers $p \le q$ satisfying $p+q \le r$,
\[
|\mathscr{L}^{p}| \le |\mathscr{L}^{r-q}|.
\]
In fact, there is an injective map $\iota:\mathscr{L}^p \to \mathscr{L}^{r-q}$ satisfying $x \le \iota(x)$ for all $x$.
\item For every positive integer $p$ less than $\frac{r}{2}$, 
\[
0 \le |\mathscr{L}^{p+1}|-|\mathscr{L}^{p}| \le \Big(|\mathscr{L}^{p}|-|\mathscr{L}^{p-1}|\Big)^{\langle p \rangle}.
\]
Equivalently, $|\mathscr{L}^0|,|\mathscr{L}^1|-|\mathscr{L}^0|,\ldots,|\mathscr{L}^{p+1}|-|\mathscr{L}^p|$  is the $h$-vector of a shellable simplicial complex.
\end{enumerate}
For undefined notions in the second statement, we refer to  \cite[Chapter II]{StanleyCC}.
\end{theorem}

The first part of Theorem \ref{MainTheoremI} settles Conjecture \ref{MainConjecture} for all $\mathscr{L}$ realizable over some field. %which was called ``intractable'' in \cite{Kung-Rota-Yan}.
We believe this to be a good demonstration of the power of the main ingredient in the proof, the decomposition theorem package for intersection complexes \cite{Beilinson-Bernstein-Deligne}.

\begin{example}
Modular geometric lattices, such as Boolean lattices or finite projective geometries,
satisfies a stronger matching property:
\begin{quote}
\emph{For every  $p$, there is an injective or surjective map $\iota: \mathscr{L}^p \to \mathscr{L}^{p+1}$ satisfying $x \le \iota(x)$.}
\end{quote}
As noted before, this implies that modular geometric lattices have the Sperner property.

%There need not be a matching between $\mathscr{L}_p$ and $\mathscr{L}_{p+1}$ for large $p$.
Dilworth and Greene   constructed in \cite{Dilworth-Greene} a configuration of $21$ points in any  $10$-dimensional projective space over a field with the property that there is no injective or surjective map 
\[
\iota: \mathscr{L}^6 \to \mathscr{L}^{7}, \qquad x \le \iota(x).
\]
Canfield  \cite{Canfield} found such ``no matching'' successive rank level sets  as above in partition lattices with sufficiently many  elements (exceeding $10^{10^{20}}$).
These geometric lattices satisfy Conjecture \ref{MainConjecture} but do not have the Sperner property.
% For geometric lattices of small rank with similar properties, see Kahn and Kung.
\end{example}

Rota  conjectured that the sizes of the rank level sets of a geometric lattice form a unimodal sequence
\cite{Rota,Rota-Harper}:
\[
|\mathscr{L}^0|\le \cdots \le |\mathscr{L}^{p-1}|\le|\mathscr{L}^p|\ge|\mathscr{L}^{p+1}|\ge \cdots \ge |\mathscr{L}^r| \ \ \text{for some $p$.}
\]
Stronger versions of this conjecture were proposed by Mason   \cite{Mason}.
The unimodality for the ``upper half'' remains as  an outstanding open problem. 

\begin{example}
Let $\lambda$ be a partition of a positive integer, which we view as a Young diagram \cite{FultonYoung}.
For example, the partition $(4,2,1)$ of $7$ corresponds to the Young diagram
\[
\yng(4,2,1)
\]
\emph{Young's lattice} associated to $\lambda$ is the graded poset  $\mathscr{L}_\lambda$  of all partitions 
whose Young diagram  fit inside  $\lambda$.   
The poset $\mathscr{L}_\lambda$ is usually not a geometric lattice, but
Bj\"orner and Ekedahl \cite{Bjorner-Ekedahl} showed  that  $\mathscr{L}_\lambda$ satisfies both conclusions of Conjecture \ref{MainConjecture} when $r$ is the number of boxes in $\lambda$:
\begin{enumerate}[(1)] \itemsep 5pt
\item For $p$ less than $\frac{r}{2}$, 
 there is an injective map $\iota:\mathscr{L}^p \to \mathscr{L}^{r-p}$ satisfying $x \le \iota(x)$ for all $x$.
\item For $p$ less than $\frac{r}{2}$, 
 there is an injective map $\iota: \mathscr{L}^p \to \mathscr{L}^{p+1}$ satisfying $x \le \iota(x)$ for all $x$.
\end{enumerate}
However, according to Stanton \cite{Stanton}, Young's lattice for the partition $(8, 8, 4, 4)$ defines a nonunimodal sequence
\begin{multline*}
\Big(|\mathscr{L}_\lambda^0|,\ |\mathscr{L}_\lambda^1|,\ |\mathscr{L}_\lambda^2|,\ \ldots,\ |\mathscr{L}_\lambda^{24}|\Big)=\\
\Big(1,\ 1,\ 2,\ 3,\ 5,\ 6,\ 9,\ 11,\ 15,\ 17,\ 21,\ 23,\ 27,\ 28,\ 31,\ 30,\ 31,\ 27,\ 24,\ 18,\ 14,\ 8,\ 5,\ 2,\ 1\Big).
\end{multline*}
Face lattices of simplicial polytopes behaves similarly, starting from dimension $20$ \cite{Billera-Lee,Bjorner}.
See \cite[Chapter 8]{Ziegler} for a discussion of unimodality in the case of polytopes.
\end{example}

\section{The graded M\"obius algebra}\label{Section2}

We use the language of matroids, and use \cite{Welsh} and \cite{Oxley} as basic references.
Let $r$ and $n$ be positive integers, and let $\mathrm{M}$ be a rank $r$ simple matroid on the ground set
\[
E=\{1,\ldots,n\}.
\]
Write  $\mathscr{L}$ for the lattice of flats of $\mathrm{M}$. 
We define a graded analogue of the M\"obius algebra for $\mathscr{L}$.

\begin{definition}
Introduce symbols $y_F$, one for each flat $F$ of $\mathrm{M}$, and construct vector spaces
\[
B^p(\mathrm{M})=\bigoplus_{F \in \mathscr{L}^p} \mathbb{Q} \hspace{0.5mm} y_F, \quad B^*(\mathrm{M})=\bigoplus_{F \in \mathscr{L}} \mathbb{Q} \hspace{0.5mm} y_F.
\]
We equip $B^*(\mathrm{M})$ with the structure of a commutative graded algebra over $\mathbb{Q}$ by setting
\[
\arraycolsep=1.1pt\def\arraystretch{1.3}
y_{F_1} y_{F_2}=\left\{\begin{array}{cl} y_{F_1 \lor F_2} & \quad \text{if $\text{rank}(F_1)+\text{rank}(F_2)=\text{rank}(F_1 \lor F_2)$,}\\ 0 & \quad \text{if $\text{rank}(F_1)+\text{rank}(F_2)>\text{rank}(F_1 \lor F_2)$.}\end{array}\right.
\]
For simplicity, we  write $y_1,\ldots,y_n$ instead of $y_{\{1\}},\ldots,y_{\{n\}}$.
\end{definition}

Maeno and Numata introduced this algebra in a slightly different form  in \cite{Maeno-Numata}, who used it to show that modular geometric lattices have the Sperner property.
Note that  $B^*(\mathrm{M})$ is generated by $B^1(\mathrm{M})$ as an algebra: If $I_F$ is any basis of a flat $F$ of $\mathrm{M}$, then
\[
y_F=\prod_{i \in I_F} y_i.
\]
Unlike its ungraded counterpart, which is isomorphic to the product of $\mathbb{Q}$'s as a $\mathbb{Q}$-algebra \cite{Solomon}, the graded M\"obius algebra $B^*(\mathrm{M})$ has a nontrivial algebra structure. Define 
\[
L=\sum_{i \in E} y_i.
\]
We deduce Theorem \ref{MainTheoremI} from the following algebraic statement.
Similar injectivity properties have appeared in the context of Kac-Moody Schubert varieties \cite{Bjorner-Ekedahl}
and toric hyperk\"ahler varieties \cite{Hausel}.

\begin{theorem}\label{MainTheoremII}
For nonnegative integer $p$ less than $\frac{r}{2}$, the multiplication map
\[
B^p(\mathrm{M}) \longrightarrow B^{r-p}(\mathrm{M}), \qquad \xi \longmapsto L^{r-2p} \ \xi
\]
is injective, when $\mathrm{M}$ is realizable over some field.
\end{theorem}

It follows that, for nonnegative integers $p \le q$ satisfying $p+q \le r$, the multiplication map
\[
B^p(\mathrm{M}) \longrightarrow B^{r-q}(\mathrm{M}), \qquad \xi \longmapsto L^{r-p-q} \ \xi
\]
is injective, when $\mathrm{M}$ is realizable over some field.
To deduce the first part of Theorem \ref{MainTheoremI} from this, consider the matrix of the  multiplication map with respect to the standard bases of the source and the target.
 Entries of this matrix are labeled by pairs of  elements of $\mathscr{L}$, and all the entries corresponding to incomparable pairs are zero.
The matrix has full rank, so there is a maximal square submatrix with nonzero determinant. In the standard expansion of this determinant, there must be a nonzero term, and the permutation corresponding to this term produces the injective map
$\iota$. 
The second part of Theorem \ref{MainTheoremI} also follows from Theorem \ref{MainTheoremII}. To see this, note that
the algebra $B^*(\mathrm{M})$ is generated by its degree $1$ elements, 
and apply Macaulay's theorem to the quotient of $B^*(\mathrm{M})$ by the ideal generated by $L$ \cite[Chapter II, Corollary 2.4]{StanleyCC}.

\begin{conjecture}
Theorem \ref{MainTheoremII} holds without the assumption of realizability.
\end{conjecture}

Let $\mathrm{M}$ be as before, and let $\overline{\mathrm{M}}$ be a simple matroid on the ground set 
\[
\overline{E}=\{0,1,\ldots,n\}.
\]
Let $\overline{\mathscr{L}}$ be the lattice of flats of $\overline{\mathrm{M}}$.
We suppose that $\mathrm{M}=\overline{\mathrm{M}}/0$, that is, $\mathrm{M}$ is obtained from $\overline{\mathrm{M}}$ by contracting the element $0$.

\begin{definition}
Introduce variables $x_{\overline{F}}$,
one for each non-empty proper flat $\overline{F}$ of $\overline{\mathrm{M}}$, and set
\[
S_{\overline{\mathrm{M}}}=\mathbb{Q}[x_{\overline{F}}]_{\overline{F} \neq \varnothing, \overline{F} \neq \overline{E},\overline{F} \in \overline{\mathscr{L}}}.
\]
The \emph{Chow ring} $A^*(\overline{\mathrm{M}})$ is the quotient of $S_{\overline{\mathrm{M}}}$
by the ideal  generated by  the linear forms
\[
\sum_{i_1 \in \overline{F}} x_{\overline{F}} - \sum_{i_2 \in \overline{F}} x_{\overline{F}},
\]
one for each pair of distinct elements $i_1$ and $i_2$ of  $\overline{E}$, and the quadratic monomials
\[
x_{\overline{F}_1}x_{\overline{F}_2},
\]
one for each pair of incomparable non-empty proper flats of $\overline{\mathrm{M}}$.
\end{definition}

The algebra $A^*(\overline{\mathrm{M}})$ and its generalizations were studied  by Feichtner and Yuzvinsky in \cite{Feichtner-Yuzvinsky}.
For every $i$ in $E$, we define an element of $A^1(\overline{\mathrm{M}})$ by setting
\[
\beta_i=\sum_{\overline{F}} x_{\overline{F}},
\]
where the sum is over all flats $\overline{F}$ of $\overline{\mathrm{M}}$ that contain $0$ and do not contain $i$.
The linear relations show that we may equivalently define
\[
\beta_i=\sum_{\overline{F}} x_{\overline{F}},
\]
where the sum is over all  flats $\overline{F}$ of $\overline{\mathrm{M}}$ that contain $i$ and do not contain $0$.
We record here three basic implications of the defining relations of $A^*(\overline{\mathrm{M}})$:
\begin{enumerate}[(1)]\itemsep 5pt
\item[(R1)] When $\overline{F}$ is a non-empty proper flat of $\overline{\mathrm{M}}$ containing exactly one of $i$ and $0$, 
\[
\beta_i \cdot x_{\overline{F}}=0.
\]
This follows from the quadratic monomial relations.
\item[(R2)] For every element $i$ in $E$, 
\[
\beta_i \cdot \beta_i=0.
\]
This follows from the previous statement.
\item[(R3)] For any two maximal chains of non-empty proper flats of $\overline{\mathrm{M}}$, say
$\{\overline{F}_k\}_{1\le k}$ and $\{\overline{G}_k\}_{1\le k}$, 
\[
\prod_{k=1}^r x_{\overline{F}_k} =\prod_{k=1}^r x_{\overline{G}_k} \neq 0. 
\]
\end{enumerate}
The proofs of (R1) and (R2) are straightforward. The proof of (R3) can be found in \cite[Section 5]{Adiprasito-Huh-Katz}.

\begin{proposition}\label{PhiHom}
There is a unique injective graded $\mathbb{Q}$-algebra homomorphism 
\[
\varphi:B^*(\mathrm{M}) \longrightarrow A^*(\overline{\mathrm{M}}), \qquad y_i \longmapsto \beta_i.
\]
\end{proposition}

\begin{proof}
First, we show that  there is a well-defined $\mathbb{Q}$-linear map
\[
\varphi:B^*(\mathrm{M}) \longrightarrow A^*(\overline{\mathrm{M}}), \qquad y_F \longmapsto \prod_{i\in I_F} \beta_i,
\]
where $I_F$ is any basis of a flat $F$ of $\mathrm{M}$.
In other words,
if  $J_F$ is any other basis of $F$,  then 
\[
\prod_{i \in I_F} \beta_i =\prod_{i \in J_F} \beta_i.
\]
Since any basis of $F$ can be obtained from any other basis of $F$ by a sequence of elementary exchanges, it is enough to check the equality in the special case when
$I_F \setminus J_F=\{1\}$ and $J_F \setminus I_F =\{2\}$.
Assuming that this is the case, we write the left-hand side of the claimed equality by
\[
\Big(\prod_{i \in I_F \cap J_F} \beta_i \Big)\Big(\sum_{\overline{G}} x_{\overline{G}}\Big),
\]
where the sum is over all non-empty proper flats $\overline{G}$ of $\overline{\mathrm{M}}$ that contain $0$ and does not contain $1$. 
The relation (R1) shows that we may take the sum 
only over those $\overline{G}$ satisfying
\[
 0 \in \overline{G}, \ \ 1 \notin \overline{G}, \ \ \text{and} \ \ I_F \cap J_F \subseteq \overline{G}.
\]
Since $I_F \cup \{0\}$ and $J_F \cup \{0\}$ are  bases of the same flat of $\overline{\mathrm{M}}$, the above condition is equivalent to
\[
 0 \in \overline{G}, \ \ 2 \notin \overline{G}, \ \ \text{and} \ \ I_F \cap J_F \subseteq \overline{G}.
\]
This proves the claimed equality, which shows that $\varphi$ is a well-defined linear map.

Second, we show that $\varphi$ is a ring homomorphism.
Given flats $F_1$ and $F_2$ of $\mathrm{M}$, we show
\[
\Big(\prod_{i \in I_{F_1}} \beta_i\Big) \Big( \prod_{i \in I_{F_2}} \beta_i \Big)=0 \ \ \text{when the rank of $F_1 \lor F_2$ is less than $|I_{F_1}|+|I_{F_2}|$.}
\]
If the independent sets $I_{F_1}$ and $I_{F_2}$ intersect, this follows from the relation (R2).
If otherwise, the condition on the rank of $F_1 \lor F_2$ implies that there are two distinct bases  of $F_1 \lor F_2$ contained in $I_{F_1} \cup I_{F_2}$, say
 \[
I_{F_1 \lor F_2} \subseteq I_{F_1} \cup I_{F_2} \ \  \text{and} \ \ J_{F_1 \lor F_2} \subseteq I_{F_1} \cup I_{F_2}.
 \]
Using the first part of the proof, once again from the relation (R2), we deduce that
\begin{align*}
 \Big(\prod_{i \in I_{F_1}} \beta_i\Big) \Big( \prod_{i \in I_{F_2}} \beta_i \Big)&=
\Big(\prod_{i \in I_{F_1\lor F_2}} \beta_i\Big) \Big( \prod_{i \in I_{F_1} \cup I_{F_2} \setminus I_{F_1\lor F_2}} \beta_i \Big)\\
&=
\Big(\prod_{i \in J_{F_1\lor F_2}} \beta_i\Big) \Big( \prod_{i \in I_{F_1} \cup I_{F_2} \setminus I_{F_1\lor F_2}} \beta_i \Big)=
0.
\end{align*}
This completes the proof that $\varphi$ is a ring homomorphism.

Third, we show that $\varphi$ is injective in degree $r$. 
Choose any ordered basis $\{i_1,\ldots,i_r\}$ of $\mathrm{M}$.
For each $q=1,\ldots,r$,  we set
\[
\overline{G}_q=\text{the closure of $\Big\{0,i_1,\ldots,i_{q-1}\Big\}$  in $\overline{\mathrm{M}}$}.
\]
We deduce from the relation (R1) that
%Then %the relation (R1) implies 
\[
\Big(\beta_{i_{1}} \cdots\beta_{i_{r-1}} \Big)\beta_{i_r} =
\Big( \beta_{i_{1}} \cdots\beta_{i_{r-1}} \Big)x_{\overline{G}_r}.
\]
Similarly, for any positive integer $q \le r$,  we have
\[
\Big( \beta_{i_{1}} \cdots\beta_{i_{q-2}}\beta_{i_{q-1}} \Big)x_{\overline{G}_q}
=\Big( \beta_{i_{1}} \cdots\beta_{i_{q-2}} \Big)x_{\overline{G}_{q-1}}x_{\overline{G}_q}=x_{\overline{G}_1}\cdots x_{\overline{G}_{q-1}}x_{\overline{G}_q}, 
\]
since $\overline{G}_{q-1}$ is the only flat of $\overline{\mathrm{M}}$ containing $\overline{G}_{q-1}$, comparable to $\overline{G}_q$, and not containing $i_{q-1}$.
Combining the above formulas, we deduce from the relation (R3) that 
\[
 \beta_{i_{1}} \cdots  \beta_{i_{r}} =x_{\overline{G}_1}\cdots x_{\overline{G}_r}  \neq 0.
\]
This proves that $\varphi$ is injective in degree $r$. 

Last, we show that  $\varphi$ is injective in any degree $q$ less than $r$.
For this we analyze the bilinear map given by the multiplication
\[
\varphi \Big(B^{q}(\mathrm{M})\Big) \times \bigoplus_{\overline{G}} \mathbb{Q}\hspace{0.5mm} x_{\overline{G}} \longrightarrow A^{q+1}(\overline{\mathrm{M}}),
\]
where the sum is over all rank $q+1$ flats $\overline{G}$ of $\overline{\mathrm{M}}$ containing $0$.
For any independent set $\{i_1,\ldots,i_{q}\}$ of $\mathrm{M}$,
we claim that,  for any  $\overline{G}$ as in the previous sentence, 
\[
\Big(\beta_{i_{1}}\cdots   \beta_{i_{q}} \Big)x_{\overline{G}}\neq 0 \ \ \text{if and only if $\overline{G}$ is the closure of $\Big\{0,i_1,\ldots,i_{q}\Big\}$  in $\overline{\mathrm{M}}$}.
\]
The ``if'' statement follows from the analysis made above. For the ``only if'' statement, suppose that the product is nonzero.
Since  $\overline{G}$ contains $0$,
it must contain $i_1,\ldots,i_{q}$ by the relation (R1).
Since $\overline{G}$ and the closure both  have the same rank, we have
\[
 \overline{G}=\text{the closure of $\Big\{0,i_1,\ldots,i_{q}\Big\}$  in $\overline{\mathrm{M}}$}.
\]
This proves the claimed equivalence, and it follows  that  the image of the basis $\{y_F\}$ of $B^q(\mathrm{M})$ under $\varphi$
is a linearly independent in $A^q(\overline{\mathrm{M}})$.
\end{proof}

\section{The simplex, the cube, and the permutohedron}

%\subsection{}
In this section, we give a toric preparation for the proof of our main result, Theorem \ref{MainTheoremII}. 
For undefined terms in toric geometry and intersection theory, we refer to \cite{FultonToric} and \cite{FultonIntersection}.
%We write $A^*(X)$ for the Chow ring of a smooth variety $X$. 
All the Chow groups and rings  will have rational coefficients.

As in the previous section, we fix a positive integer $n$ and work with the sets
\[
E=\{1,\ldots,n\} \ \ \text{and} \ \ \overline{E}=\{0,1,\ldots,n\}.
\]
Let $\mathbb{Z}^{\overline{E}}$ be the abelian group generated
by the  basis vectors $\mathbf{e}_i$ corresponding to  $i \in \overline{E}$. 
For an arbitrary subset $\overline{I} \subseteq \overline{E}$, we define
\[
\mathbf{e}_{\overline{I}}=\sum_{i \in \overline{I}} \mathbf{e}_i.
\]
We associate to $\overline{E}$ the abelian group
$N_{\overline{E}}=\mathbb{Z}^{\overline{E}}/ \mathbb{Z}\hspace{0.5mm} \mathbf{e}_{\overline{E}}$ and 
the vector space
$N_{\overline{E},\mathbb{R}}=\mathbb{R}^{\overline{E}}/ \mathbb{R}\hspace{0.5mm} \mathbf{e}_{\overline{E}}$.

\begin{enumerate}[(1)]\itemsep 5pt
\item Let $\Sigma(\mathrm{S}_n) \subseteq N_{\overline{E},\mathbb{R}}$ be the image of the normal fan of the standard $n$-dimensional simplex
\[
\mathrm{S}_n=\text{conv}\Big\{ \mathbf{e}_0, \mathbf{e}_1,\ldots, \mathbf{e}_n\Big\} \subseteq \mathbb{R}^{\overline{E}}.
\]
There are $(n+1)$ maximal cones in  $\Sigma(\mathrm{S}_n)$, one for each maximal proper subset $\overline{I}$ of $\overline{E}$:
\[
\sigma_{\overline{I}}=\text{cone}\Big\{\mathbf{e}_i \mid i \in \overline{I}\Big\} \subseteq N_{\overline{E},\mathbb{R}}.
\]
This fan defines the $n$-dimensional projective space $\mathbb{P}^n$, whose homogeneous coordinates are labeled by $i \in \overline{E}$.
\item Let $\Sigma(\mathrm{C}_n) \subseteq \mathbb{R}^E$ be the normal fan of the standard $n$-dimensional cube
\[
\mathrm{C}_n=\text{conv}\Big\{ \pm \mathbf{e}_1,\ldots,\pm \mathbf{e}_n\Big\} \subseteq \mathbb{R}^E.
\]
There are $2^n$ maximal cones in $\Sigma(\mathrm{C}_n)$, one for each subset $I$ of $E$:
\[
\sigma_I=\text{cone}\Big\{\mathbf{e}_i \mid i \in I\Big\}-\text{cone}\Big\{\mathbf{e}_i \mid i \notin I\Big\} \subseteq \mathbb{R}^E.
\]
This fan defines the product of $n$ projective lines $(\mathbb{P}^1)^n$, whose multi-homogeneous coordinates are labeled by $i \in E$.
\item Let  $\Sigma(\mathrm{P}_n) \subseteq N_{\overline{E},\mathbb{R}}$ be the image of the normal fan of the  $n$-dimensional  permutohedron
\[
\mathrm{P}_n=\text{conv}\Big\{(x_0,\ldots,x_n) \mid \text{$x_0,\ldots,x_n$ is a permutation of $0,\ldots,n$}\Big\} \subseteq \mathbb{R}^{\overline{E}}.
\]
There are $(n+1)!$ maximal cones in $\Sigma(\mathrm{P}_n)$, one for each maximal chain $\mathscr{I}$ in $2^{\overline{E}}$: 
\[
\sigma_{\mathscr{I}}=\text{cone}\Big\{\mathbf{e}_{\overline{I}} \mid \overline{I} \in \mathscr{I}\Big\}  \subseteq N_{\overline{E},\mathbb{R}}.
\]
This fan  defines the $n$-dimensional permutohedral space, denoted $X_{A_n}$.
See \cite{Batyrev-Blume} for a detailed study of $X_{A_n}$ and its analogues for other root systems.
\end{enumerate}

The inclusion  $\mathbb{Z}^E \subseteq \mathbb{Z}^{\overline{E}}$ induces an isomorphism
\[
\psi^{-1}:\mathbb{R}^E \longrightarrow N_{\overline{E},\mathbb{R}}.
\]
This identifies  the underlying vector spaces of the normal fans $\Sigma(\mathrm{S}_n)$, $\Sigma(\mathrm{P}_n)$, $\Sigma(\mathrm{C}_n)$:
\[
  \raisebox{-0.5\height}{\includegraphics{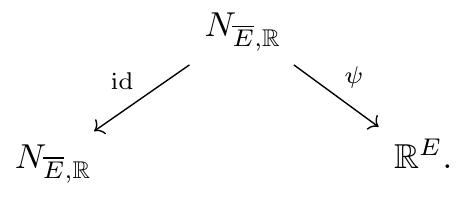}}
\]
We observe that $\textrm{id}$ and $\psi$ induce morphisms between the fans and their toric varieties
\[
  \raisebox{-0.5\height}{\includegraphics{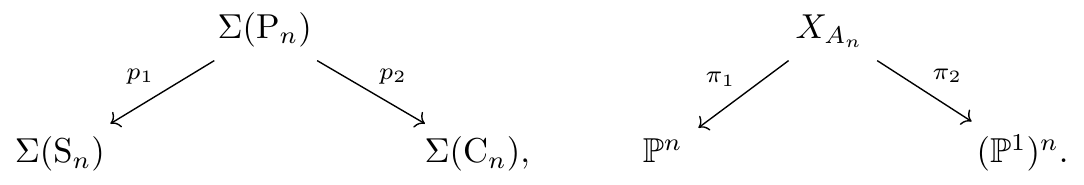}}
\]
The morphism $p_1$ is the standard barycentric subdivision. 
We check that $p_2$ is a subdivision.

\begin{proposition}
The isomorphism $\psi$ induces a morphism $p_2$.
\end{proposition} 

In other words, the image of a cone in $\Sigma(\mathrm{P}_n)$ under $\psi$ is contained in a cone in $\Sigma(\mathrm{C}_n)$.

\begin{proof}
For each $i \in E$,  define $\psi_i$ as
the composition of $\psi$ with the $i$-th projection 
\[
\psi_i=\text{proj}_i\circ \psi, \qquad \text{proj}_i:\mathbb{R}^E \longrightarrow \mathbb{R}^{\{i\}}\simeq \mathbb{R}.
\]
For any subset $\overline{I} \subseteq \overline{E}$, we have
\[
\arraycolsep=1.1pt\def\arraystretch{1.3}
\psi_i(\mathbf{e}_{\overline{I}})=\left\{\begin{array}{rl} \mathbf{e}_i&\quad \text{if $\overline{I}$ contains $i$ and does not contain $0$,} \\ -\mathbf{e}_i&\quad  \text{if $\overline{I}$ contains $0$ and does not contain $i$,} \\ 0&\quad \text{if otherwise.} \end{array}\right.
\]
It is enough to check that $\psi_i$ induces a morphism  $\Sigma(\mathrm{P}_n) \longrightarrow \Sigma(\mathrm{C}_1)$. 

Recall that any nonzero cone in the normal fan of $\textrm{P}_n$ is of the form
\[
\sigma_\mathscr{I}=\text{cone}\Big\{\mathbf{e}_{\overline{I}} \mid \overline{I} \in \mathscr{I}\Big\},
\]
where $\mathscr{I}$ is a non-empty chain in $2^{\overline{E}}$.
Viewing $\mathscr{I}$ as an ordered collection of sets, we see that
\[
\arraycolsep=1.1pt\def\arraystretch{1.3}
\left\{\begin{array}{ll}
&\text{$\psi_i(\sigma_\mathscr{I})$ is contained in the cone generated by $\mathbf{e}_i$ 
 if $i$ appears before $0$ in $\mathscr{I}$, and}\\
&\text{$\psi_i(\sigma_\mathscr{I})$ is contained in the cone generated by $-\mathbf{e}_i$ 
 if $i$ appears after $0$ in $\mathscr{I}$}.
 \end{array}\right.
 \]
 Thus the image of a cone in $\Sigma(\textrm{P}_n)$ under $\psi_i$ is contained in a cone in $\Sigma(\textrm{C}_1)$, for each $i \in E$.
\end{proof}

Geometrically, %$\pi_1$ and $\pi_2$ are compositions of blowups along smooth torus-invariant subvarieties.
$\pi_1$ is the blowup of all the torus-invariant points in $\mathbb{P}^n$, all the strict transforms of torus-invariant $\mathbb{P}^1$'s in $\mathbb{P}^n$, all the strict transforms of torus-invariant $\mathbb{P}^2$'s in $\mathbb{P}^n$, and so on. 
The map $\pi_2$ is the blowup of points $0^n$ and $\infty^n$, all the strict transforms of torus-invariant $\mathbb{P}^1$'s in $(\mathbb{P}^1)^n$ containing $0^n$ or $\infty^n$, all the strict transforms of torus-invariant $(\mathbb{P}^1)^2$'s in $(\mathbb{P}^1)^n$ containing $0^n$ or $\infty^n$, and so on.

\begin{remark}\label{PullbackRemark}
For later use, we record here a combinatorial description of the pullback of piecewise linear functions under the linear map $\psi_i=\text{proj}_i\circ \psi$:
\begin{quote} 
\emph{Let $\alpha$ be the piecewise linear function on $\Sigma(\mathrm{C}_1)$ determined by its values
\[
\alpha(\mathbf{e}_i)=1 \ \ \text{and} \ \ \alpha(-\mathbf{e}_i)=0.
\]
Then $\psi_i^*(\alpha)$ is the piecewise linear function on $\Sigma(\mathrm{P}_n)$ determined by its values
\[
\arraycolsep=1.1pt\def\arraystretch{1.3}
\psi_i^*(\alpha)(\mathbf{e}_{\overline{I}})=\left\{\begin{array}{rl} 1 &\quad  \text{if $\overline{I}$ contains $i$ and does not contain $0$,}  \\ 0&\quad \text{if otherwise.} \end{array}\right.
\]}
\end{quote}
Using the correspondence between piecewise linear functions on fans and torus-invariant divisors on toric varieties \cite[Chapter 3]{FultonToric}, the above can be used to describe the pullback homomorphism between the Chow rings
\[
\pi_2^*:A^*((\mathbb{P}^1)^n) \longrightarrow A^*(X_{A_n}).
\]
Explicitly, writing $y_i$ for the divisor of $\mathbf{e}_i$ in $(\mathbb{P}^1)^n$ and $x_{\overline{I}}$ for the divisor of  $\mathbf{e}_{\overline{I}}$ in $X_{A_n}$, 
\[
\pi_2^*(y_i)=\sum_{\overline{I}} x_{\overline{I}},
\]
where the sum is over all subsets $\overline{I} \subseteq \overline{E}$ that contain $i$ and do not contain $0$.
\end{remark}

\section{Proof of Theorem \ref{MainTheoremII}}

Let $\mathrm{M}$ be a simple matroid on $E$, and let $\overline{\mathrm{M}}$ be a simple matroid on $\overline{E}$ with $\mathrm{M}=\overline{\mathrm{M}}/0$. For simplicity, we take $\overline{\mathrm{M}}$ to be the direct sum of $\mathrm{M}$ and the rank $1$ matroid on $\{0\}$, so that $\mathrm{M}$ and $\overline{\mathrm{M}}$ share the same set of circuits.

Suppose that $\mathrm{M}$ is realizable over some field. Then $\mathrm{M}$ is realizable over some finite field \cite[Corollary 6.8.13]{Oxley}, and hence over the algebraically closed field $\overline{\mathbb{F}}_p$ for some prime number $p$.
The matroid $\overline{\mathrm{M}}$ is realizable over the same field, say by a spanning set of vectors
\[
\overline{\mathscr{A}}=\{f_0,f_1,\ldots,f_n\} \subseteq \overline{\mathbb{F}}_p^{r+1}.
\]
Dually, the realization $\overline{\mathscr{A}}$ of $\overline{\mathrm{M}}$ corresponds to an injective linear map between projective spaces
\[
i_{\overline{\mathscr{A}}}:\mathbb{P}^r \longrightarrow \mathbb{P}^n, \qquad i_{\overline{\mathscr{A}}}=[f_0:f_1:\cdots:f_n].
\]
The collection $\mathscr{A}=\{f_1,\ldots,f_n\}$ is a realization of the matroid $\mathrm{M}$.

The restriction of the torus-invariant hyperplanes of $\mathbb{P}^n$ to  $\mathbb{P}^r$ defines an arrangement of hyperplanes in $\mathbb{P}^r$, which we denote by the same symbol $\overline{\mathscr{A}}$.
We use $i_{\overline{\mathscr{A}}}$ to construct the commutative diagram
%\[
%\xymatrixcolsep{5pc}
%\xymatrixrowsep{3pc}
%\xymatrix{
%&X_\mathscr{A}\ar[r]^{j_\mathscr{A}} \ar[dl]_<<<<<<<<<<<{\pi_1^\mathscr{A}} \ar[dr]_<<<{\pi_2^\mathscr{A}}&X_{A_n} \ar[dr]^<<<<<<<<<<<<{\pi_2} \ar[dl]^<<<<<{\pi_1}&\\
%\mathbb{P}^r\ar[r]^<<<<<<<<<<<<{i_\mathscr{A}}&\mathbb{P}^n& Y_\mathscr{A} \ar[r]^<<<<<<<<<{k_\mathscr{A}}& (\mathbb{P}^1)^n,
%}
%\]
\[
  \raisebox{-0.5\height}{\includegraphics{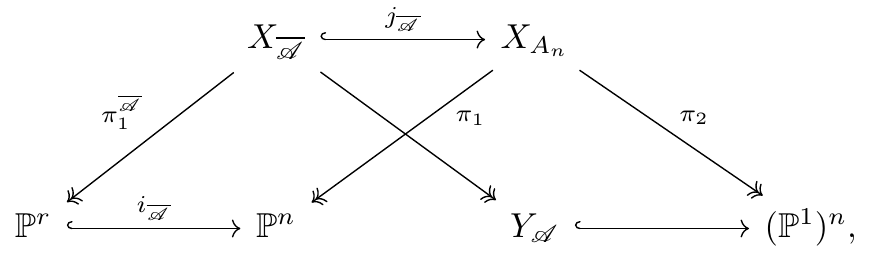}}
\]
where $X_{\overline{\mathscr{A}}}$ is the strict transform of $\mathbb{P}^r$ under  $\pi_1$ and $Y_{\mathscr{A}}$ is the image of $X_{\overline{\mathscr{A}}}$ under 
$\pi_2$. 

The induced map $\pi_1^{\overline{\mathscr{A}}}$ is the blowup of all the zero-dimensional flats of $\overline{\mathscr{A}}$,
all the strict transforms of one-dimensional flats of $\overline{\mathscr{A}}$, all the strict transforms of two-dimensional flats of $\overline{\mathscr{A}}$, and so on. 
%\begin{enumerate}[--]\itemsep 5pt
The variety $X_{\overline{\mathscr{A}}}$ is the \emph{wonderful model} of  $\overline{\mathscr{A}}$ corresponding to the maximal building set \cite{deConcini-ProcesiA}.
The variety $Y_\mathscr{A}$ is studied in \cite{Ardila-Boocher},
and its affine part centered at $\infty^n$ is the \emph{reciprocal plane} in \cite{Elias-Proudfoot-Wakefield,Proudfoot-Speyer}.
%\end{enumerate}
To apply the decomposition theorem of \cite{Beilinson-Bernstein-Deligne},
 we notice that all varieties, maps, and sheaves under consideration  may be defined over some finite extension of $\mathbb{F}_p$. 

We know that the Chow ring of $X_{\overline{\mathscr{A}}}$ is determined by the matroid $\overline{\mathrm{M}}$:
There is an isomorphism of graded algebras
\[
A^*(\overline{\mathrm{M}}) \simeq A^*(X_{\overline{\mathscr{A}}}), 
\]
where $x_{\overline{F}}$ is identified with the class of the strict transform of the exceptional divisor produced when blowing up the flat of $\overline{\mathscr{A}}$ corresponding to $\overline{F}$. See \cite[Section 1.1]{deConcini-ProcesiB}, and also \cite{deConcini-ProcesiA,Feichtner-Yuzvinsky}.
When $\overline{\mathrm{M}}$ is the Boolean matroid $\overline{\mathrm{B}}$ on $\overline{E}$, this describes the Chow ring of the permutohedral space $A^*(X_{A_n})$.
In general, the pullback homomorphism
\[
  \raisebox{-0.5\height}{\includegraphics{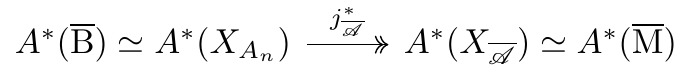}}
\]
is determined by the assignment, for non-empty proper subsets $\overline{I}$ of $\overline{E}$,
\[
\arraycolsep=1.1pt\def\arraystretch{1.3}
x_{\overline{I}} \longmapsto \left\{\begin{array}{cl} x_{\overline{I}}& \quad \text{if $\overline{I}$ is a flat of $\overline{\mathrm{M}}$,} \\ 0 & \quad \text{if $\overline{I}$ is a not flat of $\overline{\mathrm{M}}$.}\end{array}\right.
\]

Fix a prime number $\ell$ different from $p$, and consider the $\ell$-adic \'etale cohomology rings and the $\ell$-adic \'etale intersection cohomology groups  of the  varieties in the diagram above. 
These are $\overline{\mathbb{Q}}_\ell$-vector spaces of the form
\[
\mathrm{H}^*(X,\overline{\mathbb{Q}}_\ell)\vcentcolon=\mathrm{H}^*(X,\overline{\mathbb{Q}}_{\ell,X}) \ \ \text{and} \ \ \mathrm{IH}^*(X,\overline{\mathbb{Q}}_\ell)\vcentcolon=\mathrm{H}^*(X,\mathrm{IC}_X),
\]
where $\overline{\mathbb{Q}}_{\ell,X}$ and $\textrm{IC}_X$ are  constructible complexes of $\overline{\mathbb{Q}}_\ell$-sheaves on $X$ as in \cite{Beilinson-Bernstein-Deligne}.
 The blowup construction of $X_{\overline{\mathscr{A}}}$ shows that the cycle class map induces an isomorphism of commutative graded $\overline{\mathbb{Q}}_\ell$-algebras
\[
A^*(X_{\overline{\mathscr{A}}}) \otimes_\mathbb{Q} \overline{\mathbb{Q}}_\ell \simeq \mathrm{H}^{2*}(X_{\overline{\mathscr{A}}},\overline{\mathbb{Q}}_\ell),
\]
see  \cite[Appendix]{Keel}. 
For the variety $Y_{\mathscr{A}}$, which may be singular, we show in Theorem \ref{CohomologyY} that there is an isomorphism of commutative graded $\overline{\mathbb{Q}}_\ell$-algebras
\[
B^*(\mathrm{M})\otimes_\mathbb{Q} \overline{\mathbb{Q}}_\ell \simeq \textrm{H}^{2*}(Y_\mathscr{A},\overline{\mathbb{Q}}_\ell).
\]
In general, the intersection cohomology $\mathrm{IH}^*(X,\overline{\mathbb{Q}}_\ell)$ is a module over  the cohomology $\mathrm{H}^*(X,\overline{\mathbb{Q}}_\ell)$, satisfying the Poincar\'e duality and the hard Lefschetz theorems. See \cite{deCataldo-Migliorini} for an introduction and precise statements.

We obtain Theorem \ref{MainTheoremII} from the following general observation.
Let $f$ be a proper map from an $r$-dimensional smooth projective variety 
\[
f:X_1 \longrightarrow X_2,
\]
and let $L$ be a fixed ample line bundle on $X_2$. Consider the pullback homomorphism of cohomology in even degrees
\[
\mathrm{H}^{2*}(X_2,\overline{\mathbb{Q}}_\ell) \longrightarrow \mathrm{H}^{2*}(X_1,\overline{\mathbb{Q}}_\ell).
\]
The image of the pullback  is a commutative graded algebra over $\overline{\mathbb{Q}}_\ell$, denoted $B^*(f)_{\overline{\mathbb{Q}}_\ell}$:
\[
B^*(f)_{\overline{\mathbb{Q}}_\ell}=\text{im}\Big(\mathrm{H}^{2*}(X_2,\overline{\mathbb{Q}}_\ell) \longrightarrow \mathrm{H}^{2*}(X_1,\overline{\mathbb{Q}}_\ell)\Big).
\]
$B^*(f)_{\overline{\mathbb{Q}}_\ell}$ is the cyclic $\mathrm{H}^{2*}(X_2,\overline{\mathbb{Q}}_\ell)$-submodule of $\mathrm{H}^{2*}(X_1,\overline{\mathbb{Q}}_\ell)$ generated by the element $1$.

\begin{proposition}\label{PullbackProposition}
If $f$ is birational onto its image, then the multiplication  map
\[
B^p(f)_{\overline{\mathbb{Q}}_\ell} \longrightarrow B^{r-p}(f)_{\overline{\mathbb{Q}}_\ell}, \qquad \xi \longmapsto L^{r-2p} \ \xi
\]
is injective  for  every nonnegative integer $p$ less than $\frac{r}{2}$.
\end{proposition}

\begin{proof}
We reduce to the case when $f$ is surjective.
For this consider the factorization 
\[
  \raisebox{-0.5\height}{\includegraphics{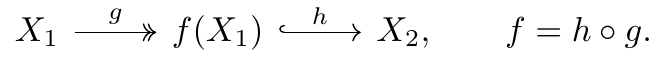}}
\]
Then $B^*(f)_{\overline{\mathbb{Q}}_\ell}$ is a subalgebra of $B^*(g)_{\overline{\mathbb{Q}}_\ell}$, and hence the statement  $(f,L)$ follows from  $(g,h^*L)$.

Suppose that $f$ is surjective. 
The decomposition theorem \cite[Section 4.3]{Beilinson-Bernstein-Deligne} says  that the intersection complex of $X_2$ appears as a direct summand  of the direct image of the constant sheaf $\overline{\mathbb{Q}}_\ell$ on $X_1$:
\[
Rf_*\overline{\mathbb{Q}}_{\ell,X_1} \simeq \mathrm{IC}_{X_2} \oplus \mathscr{C}.
\]
Taking cohomology of both sides, we obtain a splitting injection of  $\mathrm{H}^{*}(X_2,\overline{\mathbb{Q}}_\ell)$-modules
\[
\Phi:\mathrm{IH}^{*}(X_2,\overline{\mathbb{Q}}_\ell) \longrightarrow \mathrm{H}^{*}(X_1,\overline{\mathbb{Q}}_\ell).
\]
Since $\Phi$ is an isomorphism in degree $0$,
it restricts to an isomorphism of commutative algebras
\[
\text{im}\Big(\mathrm{H}^{2*}(X_2,\overline{\mathbb{Q}}_\ell) \longrightarrow \mathrm{IH}^{2*}(X_2,\overline{\mathbb{Q}}_\ell)\Big) \simeq B^*(f)_{\overline{\mathbb{Q}}_\ell}.
\]
The conclusion follows from the hard Lefschetz theorem  for $L$ on $\mathrm{IH}^{2*}(X_2,\overline{\mathbb{Q}}_\ell)$.
\end{proof}

Theorem \ref{MainTheoremII} will be deduced from the case when $f$ is the map $X_{\overline{\mathscr{A}}} \to (\mathbb{P}^1)^n$.
For each $i \in E$,  let $f_i$ be the composition of $f$ with the $i$-th projection
\[
f_i=\text{proj}_i\circ f, \qquad \text{proj}_i:(\mathbb{P}^1)^n \longrightarrow \mathbb{P}^1.
\]
As in Proposition \ref{PhiHom}, for each $i \in E$, 
 we write $\beta_i$ for the sum of $x_{\overline{F}}$ over all flats $\overline{F}$ of $\overline{\mathrm{M}}$ that contain $i$ and do not contain $0$.
 As mentioned before, the blowup construction of $X_{\overline{\mathscr{A}}}$ shows that the cycle class map induces an isomorphism of commutative graded $\overline{\mathbb{Q}}_\ell$-algebras
\[
A^*(X_{\overline{\mathscr{A}}}) \otimes_\mathbb{Q} \overline{\mathbb{Q}}_\ell \simeq \mathrm{H}^{2*}(X_{\overline{\mathscr{A}}},\overline{\mathbb{Q}}_\ell),
\]
Let $\Psi$ be the composition of isomorphisms
\[
\Psi: A^*(\overline{\mathrm{M}}) \otimes_\mathbb{Q} \overline{\mathbb{Q}}_\ell \simeq A^*(X_{\overline{\mathscr{A}}}) \otimes_\mathbb{Q} \overline{\mathbb{Q}}_\ell \simeq \mathrm{H}^{2*}(X_{\overline{\mathscr{A}}},\overline{\mathbb{Q}}_\ell),
\]
which maps $x_{\overline{F}}$ to the class of the strict transform in $X_{\overline{\mathscr{A}}}$ of   the exceptional divisor produced when blowing up the flat of $\overline{\mathscr{A}}$ in $\mathbb{P}^r$ corresponding to $\overline{F}$.

\begin{proposition}\label{PullbackBeta}
The element $\Psi(\beta_i)$   is the pullback of the class of a point in $\mathbb{P}^1$ under  $f_i$.
\end{proposition}

\begin{proof}
We factor $f$ into the composition 
\[
  \raisebox{-0.5\height}{\includegraphics{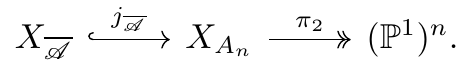}}
\]
As noted before, the pullback map associated to the inclusion $j_{\overline{\mathscr{A}}}$ satisfies
\[
\arraycolsep=1.1pt\def\arraystretch{1.3}
x_{\overline{I}} \longmapsto \left\{\begin{array}{cl} x_{\overline{I}}& \quad \text{if $\overline{I}$ is a flat of $\overline{\mathrm{M}}$,} \\ 0 & \quad \text{if $\overline{I}$ is a not flat of $\overline{\mathrm{M}}$.}\end{array}\right.
\]
Thus it is enough to prove the claim  when $X_{\overline{\mathscr{A}}} = X_{A_n}$.
This is the case when $\overline{\mathrm{M}}$ is the Boolean matroid on $\overline{E}$, and the claim in this case was proved in Remark \ref{PullbackRemark} at the level of Chow rings.
\end{proof}

Since the cohomology ring of $(\mathbb{P}^1)^n$ is generated by the pullbacks under $f_i$, Propositions \ref{PhiHom} and \ref{PullbackBeta} together show that $\Psi$ induces an isomorphism between 
$B^*(\mathrm{M})\otimes_\mathbb{Q} \overline{\mathbb{Q}}_\ell$ and $B^*(f)_{\overline{\mathbb{Q}}_\ell}$, which we denote by $\Psi'$. More precisely,
there is a commutative diagram
\[
  \raisebox{-0.5\height}{\includegraphics{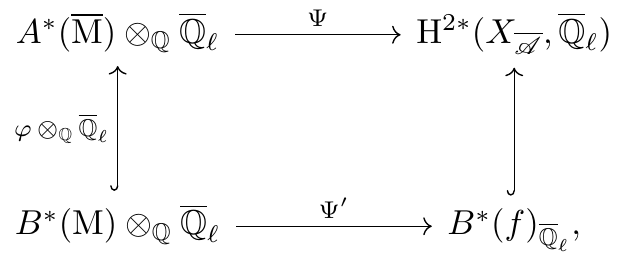}}
\]
where $\varphi$ is the injective ring homomorphism of Proposition \ref{PhiHom}.

\begin{proof}[Proof of Theorem \ref{MainTheoremII}]
It is enough to show that the multiplication map
\[
B^p(\mathrm{M}) \otimes_\mathbb{Q} \overline{\mathbb{Q}}_\ell \longrightarrow B^{r-p}(\mathrm{M}) \otimes_\mathbb{Q} \overline{\mathbb{Q}}_\ell, \qquad \xi \longmapsto L^{r-2p} \ \xi
\]
is injective. Under the isomorphism $\Psi'$, the statement to be proved translates to the conclusion of Proposition \ref{PullbackProposition} when $f$ is the map $X_{\overline{\mathscr{A}}} \to (\mathbb{P}^1)^n$.
\end{proof}

With more work, we can show that the graded M\"obius algebra of the matroid $\mathrm{M}$ is isomorphic to the cohomology ring of the variety $Y_\mathscr{A}$. 
Write $L_i$ for the first Chern class of the pullback of $\mathcal{O}(1)$ under the composition
\[
  \raisebox{-0.5\height}{\includegraphics{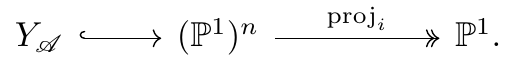}}
\]

\begin{theorem}\label{CohomologyY}
There is an isomorphism of commutative graded $\overline{\mathbb{Q}}_\ell$-algebras
\[
B^*(\mathrm{M})\otimes_\mathbb{Q} \overline{\mathbb{Q}}_\ell \simeq \textrm{H}^{2*}(Y_\mathscr{A},\overline{\mathbb{Q}}_\ell), \qquad y_i \longmapsto L_i.
\]
\end{theorem}

In what follows, we write $z_0,z_1,\ldots,z_n$ for the homogeneous coordinates of $\mathbb{P}^n$, and write
$(z_1,w_1),\ldots,(z_n,w_n)$ for the multi-homogeneous coordinates of $(\mathbb{P}^1)^n$.

\begin{proof}
Recall that $\mathrm{M}$ and $\overline{\mathrm{M}}$ share the same set of circuits. For every circuit $C$ of $\mathrm{M}$, there are nonzero constants $a_c \in \overline{\mathbb{F}}_p$, one for each element $c \in C$, such that
\[
\sum_{c \in C} a_c\hspace{0.5mm} z_c=0 \ \ \text{on the image of} \ \ i_{\overline{\mathscr{A}}}: \mathbb{P}^r \longrightarrow \mathbb{P}^n.
\]
The collection $(a_c)_{c \in C}$ is uniquely determined by the circuit $C$, up to a common multiple.

A defining set of multi-homogeneous equations of $Y_\mathscr{A}$ is explicitly described by Ardila and Boocher in \cite[Theorem 1.3]{Ardila-Boocher}: 
\[
Y_\mathscr{A}=\Bigg\{\ \sum_{c \in C} a_c \hspace{0.5mm} z_c  \Big(\prod_{d \in C \setminus c} w_{d}\Big)=0, \ \ \text{$C$ is a circuit of $\mathrm{M}$} \ \Bigg\} \subseteq (\mathbb{P}^1)^n.
\]
This shows that $Y_\mathscr{A}$ has an \emph{algebraic cell decomposition} in the sense of \cite[Section 3]{Bjorner-Ekedahl},
\[
Y_\mathscr{A}=\coprod_F \ \mathbb{A}^{\text{rank}(F)},
\]
where the disjoint union is over all flats $F$ of $\mathrm{M}$, and $\mathbb{A}^{\text{rank}(F)}$ is the intersection of $Y_\mathscr{A}$ with the affine space
\[
 \mathbb{A}^{|F|}=\Big\{\text{$w_i=0$ if and only if $i$ is not in $F$}\Big\} \subseteq (\mathbb{P}^1)^n.
\]

The existence of the cell decomposition has the following implications \cite[Theorem 3.1]{Bjorner-Ekedahl}:
\begin{enumerate}[(1)]\itemsep 5pt
\item[(CD1)] The natural map $\mathrm{H}^{2*}(Y_\mathscr{A},\overline{\mathbb{Q}}_\ell) \longrightarrow \mathrm{IH}^{2*}(Y_\mathscr{A},\overline{\mathbb{Q}}_\ell) $ is injective.
\item[(CD2)]  The dimension of  $\mathrm{H}^{2k}(Y_\mathscr{A},\overline{\mathbb{Q}}_\ell)$ is the number of $k$-dimensional cells in $Y_\mathscr{A}$ for all $k$.
\end{enumerate}
All the odd cohomology groups of $Y_\mathscr{A}$ are zero. When combined with the decomposition theorem for $X_{\overline{\mathscr{A}}} \rightarrow Y_\mathscr{A}$, the statement (CD1) shows that the pullback homomorphism in cohomology
\[
\mathrm{H}^{2*}(Y_\mathscr{A},\overline{\mathbb{Q}}_\ell) \longrightarrow  \mathrm{H}^{2*}(X_{\overline{\mathscr{A}}},\overline{\mathbb{Q}}_\ell)
\]
is injective. 
According to Proposition \ref{PullbackBeta},  the pullback of $L_i$ in the cohomology of $X_{\overline{\mathscr{A}}}$ is $\Psi(\beta_i)$, and hence the previous sentence implies that there is an injective graded ring homomorphism
\[
 B^*(f)_{\overline{\mathbb{Q}}_\ell} \longrightarrow \textrm{H}^{2*}(Y_\mathscr{A},\overline{\mathbb{Q}}_\ell), \qquad \Psi(\beta_i) \longmapsto L_i.
\]
Composing with the isomorphism $\Psi'$, we get the injective graded ring homomorphism
\[
B^*(\mathrm{M})\otimes_\mathbb{Q} \overline{\mathbb{Q}}_\ell \longrightarrow \textrm{H}^{2*}(Y_\mathscr{A},\overline{\mathbb{Q}}_\ell), \qquad y_i \longmapsto L_i.
\]
The statement (CD2) shows that the source and the target are $\overline{\mathbb{Q}}_\ell$-vector spaces of the same dimension, which is the number of flats of $\mathrm{M}$. Therefore, the map must be an isomorphism.
\end{proof}

\begin{remark}
Let $\mathrm{M}$ be a simple matroid on $E=\{1,\ldots,n\}$ with rank $r \ge 2$.  
We write ``$\text{deg}$'' for the isomorphism
\[
\text{deg}: B^r(\mathrm{M})\longrightarrow \mathbb{Q}, \qquad y_E \longmapsto 1.
\]
Let $\textrm{HR}(\mathrm{M})$ be the symmetric $n \times n$ matrix  with entries
\[
\arraycolsep=1.1pt\def\arraystretch{1.3}
\textrm{HR}(\mathrm{M})_{ij} = \left\{\begin{array}{cl} 0 & \quad \text{if $i=j$,}\\ b_{ij}(\mathrm{M}) &\quad \text{if $i \neq j$,}\end{array}\right.
\]
where $b_{ij}(\mathrm{M})$ is the number of bases of $\mathrm{M}$ containing $i$ and $j$.
The matrix $\textrm{HR}(\mathrm{M})$ represents the Hodge-Riemann form
\[
B^1(\mathrm{M}) \times B^1(\mathrm{M}) \longrightarrow \mathbb{Q}, \qquad  (\xi_1,\xi_2) \longmapsto \text{deg}( L^{r-2}\ \xi_1 \ \xi_2),
\]
with respect to the standard basis $y_1,\ldots,y_n$.
It can be shown that the matrix $\textrm{HR}(\mathrm{M})$ has exactly one positive eigenvalue \cite{Huh-Wang}.

Consider the restriction of $\textrm{HR}(\mathrm{M})$ to the three dimensional subspace of $B^1(\mathrm{M})$ spanned by $y_i$, $y_j$, and $L$. 
The one positive eigenvalue condition says that the determinant of the resulting symmetric $3 \times 3$ matrix is nonnegative, and this implies
\[
2> b(\mathrm{M}) b_{ij}(\mathrm{M})/b_i(\mathrm{M}) b_j(\mathrm{M}),
\]
where $b(\mathrm{M})$ is the number of bases of $\mathrm{M}$ and $b_i(\mathrm{M})$ is the number of bases of $\mathrm{M}$ containing $i$.
More detailed arguments will be given in \cite{Huh-Wang}.
\begin{quote}
\emph{Question: How large can the ratio $b(\mathrm{M}) b_{ij}(\mathrm{M})/b_i(\mathrm{M}) b_j(\mathrm{M})$ be?}
\end{quote}
For graphic matroids, the work of Kirchhoff on electric circuits shows that the ratio is bounded above by $1$, see \cite{Feder-Mihail}.
In other words, for a randomly chosen spanning tree of a graph, the presence of an edge can only make any other edge less likely.
It was once conjectured that this is the case for all matroids, but Seymour and Welsh found an example with  the ratio $ \simeq 1.02$ \cite{Seymour-Welsh}.
\end{remark}

\subsection*{Acknowledgements}
We thank Petter Br\"and\'en, Jeff Kahn, Satoshi Murai, Yasuhide Numata, Nick Proudfoot, Dave Wagner, and Geordie Williamson for  helpful conversations.
Special thanks go to two anonymous referees, who made very useful suggestions.
This research started while Botong Wang was visiting Korea Institute for Advanced Study in summer 2016. We thank KIAS for excellent working conditions.
June Huh was supported by a Clay Research Fellowship and NSF Grant DMS-1128155.

\end{document}